\documentclass[11pt,reqno]{amsart}

\usepackage[dvipsnames]{xcolor}
\usepackage{tikz}
\usepackage{stmaryrd}
\usetikzlibrary{matrix,arrows,decorations.pathmorphing}
\usepackage{graphicx}
\usepackage{amssymb}
 \usepackage[stable]{footmisc}
\usepackage{amsmath,mathtools}
\usepackage{fullpage}
\usepackage{caption}
\usepackage{amsthm}
\usepackage{mathrsfs}
\usepackage{thmtools}
\usepackage{blkarray}
\usepackage{multirow}
\usepackage{picinpar} 
\usepackage{tikz-cd}
\usepackage{color}
\usepackage{verbatim}
\usepackage{hyperref}
\hypersetup{colorlinks=true, linkcolor=black, citecolor = OliveGreen, urlcolor = black}
\usepackage{amssymb}
\usepackage{amsmath}
\usepackage[shortlabels]{enumitem}
\usepackage{colonequals}
\usepackage{faktor}

\usepackage{color}
\usepackage{mathrsfs}
\usepackage[all]{xy}
\usepackage{comment}
\usepackage{mathpazo}
\usepackage{euler}
\usepackage[noabbrev,capitalise]{cleveref}

\usepackage[margin=1.in]{geometry}
\tikzset{
  closed/.style = {decoration = {markings, mark = at position 0.5 with { \node[transform shape, xscale = .8, yscale=.4] {/}; } }, postaction = {decorate} },
  open/.style = {decoration = {markings, mark = at position 0.5 with { \node[transform shape, scale = .7] {$\circ$}; } }, postaction = {decorate} }
}

\newcommand{\mZ}{\mathbb{Z}}

\newcommand{\mQ}{\mathbb{Q}}
\newcommand{\mF}{\mathbb{F}}
\newcommand{\mC}{\mathbb{C}}

\newcommand{\cD}{\mathcal{D}}

\newcommand{\cA}{\mathcal{A}}

\newcommand{\cO}{\mathcal{O}}

\newcommand{\cB}{\mathcal{B}}

\usepackage[OT2,T1]{fontenc}
\DeclareSymbolFont{cyrletters}{OT2}{wncyr}{m}{n}
\DeclareMathSymbol{\Sha}{\mathalpha}{cyrletters}{"58}

\DeclareMathSymbol{\Sha}{\mathalpha}{cyrletters}{"58}

\newcommand{\brk}[1]{ \left\lbrace #1 \right\rbrace }

\newcommand{\sF}{{\mathscr F}}

\newcommand{\uu}{\underline{u}}

\RequirePackage{mathrsfs} 

\numberwithin{equation}{section}
\newtheorem{thmx}{Theorem}

\numberwithin{equation}{subsection}
\newtheorem{theorem}[subsection]{Theorem}

\newtheorem{lemma}[subsection]{Lemma}

\newtheorem{proposition}[subsection]{Proposition}

\theoremstyle{definition}
\newtheorem{definition}[subsection]{Definition}

\theoremstyle{remark}

\newtheorem{remark}[subsection]{Remark}

\numberwithin{equation}{section} \numberwithin{figure}{section}

\DeclareMathOperator{\gal}{Gal}

\DeclareMathOperator{\Lie}{Lie}
\DeclareMathOperator{\Hom}{Hom}

\DeclareMathOperator{\Spf}{Spf}

\newcommand{\cdef}[1]{\textsf{\textit{#1}}}

\renewcommand{\leq}{\leqslant}
\renewcommand{\le}{\leqslant}
\renewcommand{\geq}{\geqslant}
\renewcommand{\ge}{\geqslant}

\DeclareMathOperator{\Cz}{Cz}

\newcommand\nc{\newcommand}
\nc\mf\mathfrak
\nc\mc\mathcal
\nc\mb\mathbb
\nc\msf\mathsf

\makeatletter
\@namedef{subjclassname@1991}{\emph{2020} Mathematics Subject Classification}
\makeatother

\newcommand{\Z}{{\mathbb Z}}
\newcommand{\Q}{{\mathbb Q}}
\newcommand{\C}{{\mathbb C}}

\newcommand{\F}{{\mathbb F}}
\newcommand{\N}{{\mathbb N}}
\newcommand{\cE}{{\mathcal E}}

\newcommand{\Kbar}{{\overline{K}}}

\makeatletter
\let\@wraptoccontribs\wraptoccontribs
\makeatother

\begin{document}

\title{Ramification of $p$-power torsion points of formal groups}

\author{Adrian Iovita}
\address{Adrian Iovita \\
	Concordia University \\
	Department of Mathematics and Statistics \\ 
	Montr\'eal, Qu\'ebec, Canada and 
Dipartimento di Matematica \\
Universita degli Studi di Padova \\
Padova, Italy
}
\email{adrian.iovita@concordia.ca}

\author{Jackson S. Morrow}
\address{Jackson S. Morrow \\
	Department of Mathematics \\
	University of California, Berkeley \\
	749 Evans Hall, Berkeley, CA 94720}
\email{jacksonmorrow@berkeley.edu}

\author{Alexandru Zaharescu}
\address{Alexandru Zaharescu\\
Department of Mathematics \\
University of Illinois Urbana-Champaign \\
1409 West Green Street, Urbana, IL 61801, USA and
''Simion Stoilow'' Institute of Mathematics of the Romanian Academy\\
P.O. Box 1-764, RO-014700 Bucharest, Romania.
}
\email{zaharesc@illinois.edu}


\date{\today}

\begin{abstract}
Let $p$ be a rational prime, let $F$ denote a finite, unramified extension of $\mathbb{Q}_p$, let $K$ be the completion of the maximal unramified extension of $\mathbb{Q}_p$, and let $\overline{K}$ be some fixed algebraic closure of $K$. 
Let $A$ be an abelian variety defined over $F$, with good reduction, let $\mathcal{A}$ denote the N\'eron model of $A$ over ${\rm Spec}(\mathcal{O}_F)$, and let $\widehat{\mathcal{A}}$ be the formal completion of $\mathcal{A}$ along the identity of its special fiber, i.e. the formal group of $A$.

In this work, we prove two results concerning the ramification of $p$-power torsion points on $\widehat{\mathcal{A}}$. 
One of our main results describes conditions on $\widehat{\mathcal{A}}$, base changed to $\text{Spf}(\mathcal{O}_K) $, for which the field $K(\widehat{\mathcal{A}}[p])/K$ is a tamely ramified extension where $\widehat{\mathcal{A}}[p]$ denotes the group of $p$-torsion points of $\widehat{\mathcal{A}}$ over $\mathcal{O}_{\overline{K}}$.  
This result generalizes previous work when $A$ is $1$-dimensional and work of Arias-de-Reyna when $A$ is the Jacobian of certain genus 2 hyperelliptic curves. 
\end{abstract}
\maketitle

 \section{\bf Introduction}\label{sec:intro}
In this work, we are interested in studying the ramification behaviour of the $p$-torsion points of the formal group associated to an abelian variety over an unramified local field.

Let $p$ be a rational prime, let $F$ denote a finite, unramified extension of $\Q_p$, let $K$ be the completion of the maximal unramified extension of $\Q_p$, let $\Kbar$ be some fixed algebraic closure of $K$, and let $\cO \coloneqq \mathcal{O}_{\Kbar}$. 
Let $A$ be an abelian variety defined over $F$, with good reduction, let $\cA$ denote the N\'eron model of $A$ over ${\rm Spec}(\cO_F)$, and let $\widehat{\cA}$ be the formal completion of $\cA$ along the identity of its special fiber, i.e. the formal group of $A$.

To state these results,  we need the following definitions. 
In \cite{fontaine:differentials}, Fontaine studied the $\cO$-module $\Omega:=\Omega^1_{\cO/\cO_K}\cong \Omega^1_{\cO/\cO_F}$ of K\"ahler differentials of $\cO$ over $\cO_K$, or over $\cO_F$. 
The $\cO$-module $\Omega$ is a torsion and $p$-divisible $\cO$-module, with a semi-linear action of $G_F$. 
Let $d\colon\cO\to \Omega$ denote the canonical derivation, which is surjective. 
We denote by $\cO^{(1)}:=\ker(d)$, the kernel of $d$, which is an $\cO_K$-sub-algebra of $\cO$. 
Membership of an element $x$ of $\cO$ inside of $\cO^{(1)}$ reflects ramification properties of $x$ (see e.g., \autoref{defn:delta} and \autoref{lemma:propertiesofdelta}). 

By using previous work of the authors \cite{IMZ:Fontaine}, we are able to prove that the $\mathcal{O}^{(1)}$-points of the Tate module of $\cA$ are trivial, which implies that following theorem. 

 \begin{thmx}\label{xthm:main1}
Let $A$ be an abelian variety over $F$ with good reduction. Then there is $n_0\ge 1$ such that 
for every $m\ge n_0$ and $0\neq P\in \widehat{\cA}[p^m](\cO)\setminus \widehat{\cA}[p^{n_0 -1 }](\cO)$, we have $P\notin\widehat{\cA}(\cO^{(1)})$.
\end{thmx}

For a more concretely description of what this means in terms of the coordinates of the torsion point $P$, we refer the reader to  \autoref{remark:ram}.

Our second result gives conditions on $\widehat{\cA}$ for which one may take $n_0 = 1$ in \autoref{xthm:main1}. 
More precisely, we describe a condition on a formal group $\sF$ of dimension $g$ over $\Spf(\cO_K)$ which implies that $0\neq P = (x_1,\dots,x_g) \in \sF[p](\cO)$, the field of definition $K(P)/K$ is tamely ramified. 
The condition is discussed in Section \ref{sec:ramification} and is related to a symmetric formal group law from \cite{arias}. 

\begin{thmx}\label{xthm:tamelyramifiedtorsion}
Let $\sF$ be a strict (\autoref{defn:strict}) formal group of dimension $g$ over $\Spf(\mathcal{O}_K)$.
For $0\neq P = (x_1,\dots,x_g) \in \sF[p](\cO)$, the field of definition $K(P)/K$ is tamely ramified and $\cO_{K(P)} \cong \cO_K[x_1,\dots,x_g]$. 
Moreover, $K(\sF[p])/K$ is tamely ramified. 
\end{thmx}

\subsection{Related results}
In \cite[Section 1]{Serre:PropertiesGaloisEllipticCurves}, Serre showed that for $E/\mQ_p$ an elliptic curve with good supersingular reduction, the field extension $\mQ_p(E[p])/\mQ_p$ is tamely ramified, and his proof relies on a detailed study of the formal group attached to $E$. In particular, Serre explicitly determined the $p$-adic valuation of the points on $E[p]$ (note that when $E/\mQ_p$ has good supersingular reduction we have that $E[p]\cong \widehat{\cE}[p]$ where $\widehat{\cE}$ is the formal group of the N\'eron model of $E$), which allowed him to embed $E[p]$ into a certain vector space on which the wild inertia group acts trivially.

In her thesis, Arias-de-Reyna generalized this approach, and in \cite[Theorem 3.3]{arias}, she showed that if there exists a positive rational number such that for all $0\neq P = (x_1,\dots,x_g) \in \widehat{\cA}[p](\cO)$, the minimum of $p$-adic valuation of the coordinates of $P$ equals $\alpha$, then the action of wild inertia on $\widehat{\cA}[p]$ is trivial, and so $K(\widehat{\cA}[p])/K$ is tamely ramified. 
She goes on to define the notion of a symmetric formal group law on a formal group of dimension 2, and then proves \cite[Theorem 4.15]{arias} that if formal group of dimension 2 has a symmetric formal group law and height 4, then the above statement about the $p$-adic valuation of the $p$-torsion points holds. 
Later in \cite[Theorem 5.9]{arias}, she identifies a family of genus 2 curves whose Jacobians have associated formal groups with a symmetric formal group law and height 4, and hence their $p$-torsion defines a tamely ramified extension.  We also mention work of Rosen and Zimmerman \cite{RosenZimmerman:TorsionFormal}, in which the authors study the Galois group of $K(\sF[p^n])$ where $\sF$ is a generic commutative formal group of dimension 1 and height $h$. 

In the global setting i.e., when working over a number field $F/\mQ$, Coleman \cite{coleman1987ramified} studied the ramification properties of torsion points on abelian varieties in relation to the Manin--Mumford conjecture. 
More precisely, he conjectured (\textit{loc.~cit.~}Conjecture B) that for a smooth, projective, geometrically integral curve $C/F$ and any Galois stable torsion packet $T$ in $C(\overline{F})$, the field $F(T)/F$ is unramified at a certain prime $\mathfrak{p}$ of $F$. 
Coleman proved this conjecture when $\mathfrak{p}$ is large enough, and using work of Bogomolov, he provided a new proof of the Manin--Mumford conjecture. 

We conclude by noting that our \autoref{xthm:tamelyramifiedtorsion} generalizes \cite[Theorem 4.15]{arias} and theoretically provides examples of abelian varieties of arbitrary dimension for which the extension $K(\widehat{\cA}[p])/K$ is tamely ramified.

\subsection{Outline of paper}
In Section \ref{sec:Fontaine}, we recall the definition of the Fontaine integral, our previous work on the kernel of the Fontaine integral \cite{IMZ:Fontaine}, and a different perspective on the Fontaine integral via work of Wintenberger.  
In Section \ref{sec:differentperspective}, we prove \autoref{xthm:main1}. 
We conclude in Section \ref{sec:ramification} with the definition of a strict formal group and our proof of \autoref{xthm:tamelyramifiedtorsion}. 

 \subsection{Conventions}\label{subsec:conventions}
We establish the following notations and conventions throughout the paper.

\subsubsection*{Fields}
Fix a rational prime $p>2$.  
Let $K$ denote the completion of maximal unramified extension of $\Q_p$, let $\overline{K}$ be a fixed algebraic closure of $K$, and let $\C_p$ denote the completion of $\overline{K}$ with respect to the unique extension $v$ of the $p$-adic valuation on $\Q_p$ (normalized such that $v(p) = 1$). 
For a tower of field extensions $\Q_p\subset F\subset K$, we denote by $G_K$ and respectively $G_F$ the absolute Galois groups of $K$ and $F$ respectively. 
We denote $\cO:=\cO_{\Kbar}$.

\subsubsection*{Abelian varieties}
We will consider an abelian variety $A$ defined over some  subfield $F\subset K$ such that $[F:\Q_p]<\infty$, with good reduction over 
$F$.
Let $\cA$ denote the N\'eron model of $A$ over ${\rm Spec}(\cO_F)$ and also denote by $\widehat{\cA}$ the formal completion of $\cA$ along the identity of its special fiber, i.e. the formal group of $A$.
We note that the formation of N\'eron models commutes with unramified base change. 
We will denote the Tate module of $A$ (resp.~the N\'eron model $\cA$ of $A$) by $T_p(A)$ (resp.~$T_p(\cA)$).
We note that $T_p(A)\cong T_p(\cA)$ as $G_F$-modules.

\subsubsection*{Formal groups}
We will let $\sF$ denote a formal group over $\Spf(\cO_K)$. 
Recall that $\widehat{\cA}$ is a formal group of dimension $\dim (A)$ and of height $h$ which satisfies $\dim(A) \leq h \leq 2\dim(A)$. 
We refer the reader to \cite{Hazewinkel:FormalGroups} for an extensive treatment of formal groups and to \cite[Chapter 4.2 and Chapter 5]{AriasDeReyna:Thesis} for a more concise treatment. 

 \section{\bf Fontaine integration for abelian varieties with good reduction}\label{sec:Fontaine}
In this section, we recall the construction of the Fontaine integration as well as our previous work concerning the kernel of the Fontaine integral.

\subsection*{The differentials of the algebraic integers}
First, we recall for the reader's convenience the notation established above. 
Let $K$ denote the maximal unramified extension of $\Q_p$, let $\overline{K}$ be a fixed algebraic closure of $K$, and let $\C_p$ denote the completion of $\overline{K}$. 
Let $G_K$ denote the absolute Galois group of $K$. 
We denote $\cO:=\cO_{\Kbar}$. Fix a finite extension $F$ of $\Q_p$ in $K$.
For a $G_K$-representation $V$, the $n$-th Tate twist of $V$ is denoted by $V(n)$, which is just the tensor product of $V$ with the $n$-fold product of the $p$-adic cylcotomic character $\mQ_p(1)$. 

In \cite{fontaine:differentials}, Fontaine studied a fundamental object related to these choices, namely the $\cO$-module $\Omega:=\Omega^1_{\cO/\cO_K}\cong \Omega^1_{\cO/\cO_F}$ of K\"ahler differentials of $\cO$ over $\cO_K$, or over $\cO_F$. 
The $\cO$-module $\Omega$ is a torsion and $p$-divisible $\cO$-module, with a semi-linear action of $G_F$. 
Let $d\colon\cO\to \Omega$ denote the canonical derivation, which is surjective.

Important examples of algebraic differentials arise as follows:~Let $(\varepsilon_n)$ denote a compatible sequence of primitive $p$th roots of unity in $\overline{K}$. 
Then 
\[
\frac{d\varepsilon_n}{\varepsilon_n} = d (\log \varepsilon_n) \in \Omega \quad \mbox{ and } \quad  p\left(\frac{d\varepsilon_{n+1}}{\varepsilon_{n+I}}\right)=\frac{d\varepsilon_n}{\varepsilon_n}. 
\]

Next, we recall a theorem of Fontaine.

\begin{theorem}[\protect{\cite[Th\'eor\`eme 1']{fontaine:differentials}}]\label{thm:Fontainedifferentials}
Let $(\varepsilon_n)$ denote a compatible sequence of primitive $p$th roots of unity in $\overline{K}$. 
For $\alpha \in K$, write $\alpha = a/p^r$ for some $a\in \mathcal{O}$. 
The morphism $\xi\colon \overline{K}(1) \to \Omega$ defined by 
\[
\xi(\alpha \otimes (\varepsilon_n)_n) = a \frac{d\varepsilon_r}{\varepsilon_r}
\]
is surjective and $G_K$-equivariant with kernel 
\[
\ker(\xi) = \underline{a}_{K} \coloneqq \brk{x\in \overline{K} : v(x) \geq - \frac{1}{p-1}}.
\]
Moreover, $\Omega\cong{\overline{K}(1)}/{\underline{a}_K(1)} \cong (\overline{K}/\underline{a}_K)(1)$ and $V_p(\Omega) = \Hom_{\mZ_p}(\mQ_p,\Omega) \cong \mC_p(1)$.
\end{theorem}

\autoref{thm:Fontainedifferentials} implies the following:
\[
T_p(\Omega)\otimes_{\Z_p}\Q_p:=\left(\varprojlim_n \Omega[p^n]\right) \otimes_{\Z_p}\Q_p\cong \left(\varprojlim\left(\Omega\stackrel{p}{\leftarrow}\Omega\stackrel{p}{\leftarrow}\cdots\stackrel{p}{\leftarrow}\Omega\cdots\right)\right)\otimes_{\Z_p}\Q_p\cong\C_p(1)
\]
as $G_F$-modules.

We denote by $\cO^{(1)}:=\ker(d)$, the kernel of $d$, which is an $\cO_K$-sub-algebra of $\cO$. 
Indeed, if $a,b\in \cO^{(1)}$, then $d(ab)=ad(b)+bd(a)=0$, and so $ab\in \cO^{(1)}$. 
In order to better understand $\cO^{(1)}$, we recall a construction from the first and last author \cite{iovita_zaharescu}. 

\begin{definition}\label{defn:delta}
Let $a\in \cO$. Let $L/K$ be a finite extension which contains $a$, let $\pi$ be a uniformizer of $L$, and let $f\in \cO_K[x]$ be such that $a = f(\pi)$. Then, define 
\[
\delta(a) \coloneqq \min \left( v\left( \frac{f'(\pi)}{\mathcal{D}_{L/K}}\right),0 \right) 
\]
where $\mathcal{D}_{L/K}$ denotes the different ideal of $L/K$. 
Note that $\delta$ does not depend on $\pi$, $f$, or $F$, and so it defines a function $\delta\colon \cO \to (-\infty,0]$. 
\end{definition}

\begin{lemma}[Properties of $\delta$]\label{lemma:propertiesofdelta}
The function $\delta$ from \autoref{defn:delta} satisfies the following properties.
\begin{enumerate}
\item If $a,b\in \cO$, then $\delta(a + b) \geq \min(\delta(a),\delta(b))$, and if $\delta(a) \neq \delta(b)$, then we have equality.
\item If $a,b\in \cO$, then $\delta(ab) \geq \min(\delta(a) + v(b),\delta(b) + v(a))$. 
\item If $f\in \cO_K[x]$ and $\alpha \in \cO$, then $\delta(f(\alpha)) = \min(v(f'(\theta)) + \delta(\theta) ,0)$.
\item If $x,y \in $, then $xdy = 0$ if and only if $v(x) + \delta(y) \geq 0$.
\item For $a\in \cO$, $\delta(a) = 0$ if and only if $a\in \cO^{(1)}$. 
\item The formula $\delta(a db) \coloneqq \min(v(a) + \delta(b),0)$ is well-defined and give a map $\delta\colon \Omega \to (-\infty,0]$, which makes the obvious diagram commutative. 
\end{enumerate}
\end{lemma}

We will use the follow properties of $\delta$ in our study of the Fontaine integral. 

\begin{lemma}[\protect{\cite[Lemma 2.2]{iovita_zaharescu}}]\label{lemma:relateddelta}
Let $a,b \in \cO$ be such that $\delta(a)\leq \delta(b)$. Then there exists $c\in \cO_{K[a,b]}$ such that $c d a = d b$. 
\end{lemma}

\begin{proposition}[\protect{\cite[Theorem 2.2]{iovita_zaharescu}}]\label{prop:deeplyramified}
Let $L/K$ be an algebraic extension. Then $L$ is deeply ramified  (\textit{loc.~cit.~}Definition 1.1) if and only if $\delta(\cO_L)$ is unbounded. 
\end{proposition}

\subsection*{The definition of Fontaine's integration}\label{sec:Fontaineintegral}
We are now ready to define Fontaine's integration. 
Let $ H^0(\cA, \Omega^1_{\cA/\cO_F})$ and respectively $\Lie(\cA)(\cO_F)$ denote the $\cO_F$-modules of invariant differentials on $\cA$ and respectively its Lie algebra. 
Note that $\omega\in  H^0(\cA, \Omega^1_{\cA/\cO_F})$ being invariant implies that $(x\oplus_{
\cA} y)^*(\omega) = x^*(\omega) + y^*(\omega)$ and $[p]^*(\omega) = p\omega$ where $\oplus_{
\cA}$ is the group law in $A(\Kbar)$. 

\begin{definition}\label{defn:Fontaineintegral}
Let $\underline{u}=(u_n)_{n\in \N}\in T_p(A)$ and $\omega\in H^0(\cA, \Omega^1_{\cA/\cO_F})$.
Each $u_n\in \cA(\cO)$ corresponds to a morphism $u_n\colon {\rm Spec}(\cO)\to \cA$, and hence we can pullback $\omega$ along this map giving us a K\"ahler differential $u_n^*(\omega)\in \Omega$. 
The sequence $\left(u_n^*(\omega)\right)_{n\geq 0}$ is a sequence of differentials in $\Omega$ satisfying
$pu_{n+1}^\ast(\omega)=u_n^\ast(\omega)$, and hence defines an element in $V_p(\Omega)\cong \C_p(1)$.

The \cdef{Fontaine integration map} 
\[
\varphi_\cA\colon T_p(A)\to \Lie(\cA)(\cO_F)\otimes_{\cO_F}\C_p(1)
\] 
is a non-zero, $G_F$-equivariant map defined by
\[
\varphi_\cA(\underline{u})(\omega):=\left(u_n^\ast(\omega)\right)_{n\ge 0}\in V_p(\Omega)\cong \C_p(1).
\]
%
%
\end{definition}

\begin{remark}\label{rem:differentdefn}
Using \autoref{thm:Fontainedifferentials} and the function $\delta$ from \autoref{defn:delta}, we can give an alternative description of the Fontaine integration map. 
Let $\underline{u}=(u_n)_{n\geq 0}\in T_p(A)$ and $\omega\in H^0(\cA, \Omega^1_{\cA/\cO_F})$.
Each $u_n\in \cA(\cO)$ corresponds to a morphism $u_n\colon {\rm Spec}(\cO)\to \cA$, and hence we can pullback $\omega$ along this map giving us a K\"ahler differential $u_n^\ast(\omega)\in \Omega$. 

For every $n \geq 0$, there is a maximal $m(n) \geq 0$ such that $u_n^*(\omega) = \alpha_n (d\varepsilon_{m(n)}/\varepsilon_{m(n)})$ with $\alpha_n \in \cO$ where $\varepsilon_{m(n)}$ is some primitive $p^{m(n)}$-th root of unity. 
To see this, we first note that 
\[
\delta\left(\frac{d\varepsilon_r}{\varepsilon_r}\right) = -r - \frac{1}{p^r(p-1)} 
\]
for any primitive $p^r$-th root of unity. This result follow from the definition of $\delta$ and a result of Tate \cite[Proposition 5]{tate:pdiv} on the valuation of the different ideal of $K(\varepsilon_r)/K$. 
By taking $m(n) = -[\delta(u_n^\ast(\omega))]$ where $[x]$ denotes the greatest integer of the real number $x$, we can use  \autoref{lemma:propertiesofdelta}.(6) and \autoref{lemma:relateddelta} to deduce the above equality. 

Now using \autoref{thm:Fontainedifferentials}, we have that 
\[
\varphi_{\cA}(\underline{u})(\omega) = \lim_{n\to \infty}p^{n-m(n)}\alpha_n \in \C_p. 
\]
Moreover, using the definition of $\delta$ and this above interpretation, we can see that if $\uu \in T_p(A)^{G_K}$ (i.e., if $\uu$ is an unramified path), then $\varphi_{\cA}(\underline{u})(\omega) = 0$. Indeed, it is clear from the definition of $\delta$ that $m(n) = 0$. 
\end{remark}

\subsection*{The kernel of the Fontaine integral}
In \cite{IMZ:Fontaine}, we studied the kernel of $\varphi_A$. 
As noted in \autoref{rem:differentdefn},  we have that $T_p(A)^{G_K}$ lies in $\ker(\varphi_A)$, and in \cite[Theorem 4.5, Theorem A.4]{IMZ:Fontaine}, we showed that $T_p(A)^{G_K} = \ker(\varphi_A)$. 
In proving these results, we determined the kernel of the Fontaine integral when restricted to the Tate module of the formal group of $A$. This result will play a role later on, and so we present it below.

\begin{theorem}[\protect{\cite[Theorem 5.5]{IMZ:Fontaine}}]\label{thm:kernelformal}
Let $A$ be an abelian variety over $F$ with good reduction, let $\cA$ denote its N\'eron model, and let $\widehat{\cA}$ be the formal group of $A$. 
The Fontaine integral restricted to the Tate module of $\widehat{\cA}$ is injective i.e., $\ker((\varphi_A)_{|T_p(\widehat{\cA})}) = 0$. 
\end{theorem}

 \subsection*{\bf Another point of view on the Fontaine integration map}\label{sec:another}
 In this subsection, we give another perspective on the Fontaine integration map, which will naturally lead us towards an application of \autoref{thm:kernelformal}. 
 
We keep all the notations from the previous sections and Subsection \ref{subsec:conventions}. 
Recall that we let $\cO:=\cO_{\Kbar}$ and we have the $\cO$-module $\Omega:=\Omega^1_{\cO/\cO_K}$ with its canonical derivation $d\colon \cO\to\Omega$.
Note that $d$ is surjective and $\Omega$ is $p$-divisible, and let us denote $\cO^{(1)}:=\ker(d)$.

\begin{lemma}\label{lemma:squarezero}
Let $A_{\rm inf}^{(1)}$ denote the $p$-adic completion of $\cO^{(1)}$.
Then, the exact sequence of $G_F$-modules
\[
0\to \cO^{(1)}\to \cO\stackrel{d}{\to}\Omega\to 0
\]
induces another exact sequence:
\[
0\to T_p(\Omega)\to A_{\rm inf}^{(1)}\stackrel{\gamma}{\to}\cO_{\C_p}\to 0,
\]
where $\gamma$ is an $\cO_F$-algebra homomorphism and $T_p(\Omega)$ is seen as an ideal of 
$A_{\rm inf}^{(1)}$ of square $0$.
\end{lemma}

\begin{proof}
The statement follows from \cite[Lemme 3.8]{Colmez:BdR} and also from \cite[Corollary 1.1]{iovita_zaharescu}, but we present another proof below.

We consider the diagram
\[
\begin{tikzcd}[row sep = 1.2em]
0 \arrow{r} & \cO^{(1)} \arrow{r} \arrow{d}{p^n}&\cO \arrow{r}{d}\arrow{d}{p^n} &\Omega \arrow{r}\arrow{d}{p^n} &0\\
0 \arrow{r} & \cO^{(1)} \arrow{r} &\cO \arrow{r}{d} &\Omega \arrow{r} &0.
\end{tikzcd}
\]
The snake lemma gives the exact sequence of $G_F$-modules:
\[
0\to \Omega[p^n]\to \cO^{(1)}/p^n\cO^{(1)}\to \cO/p^n\cO\to 0.
\]
By taking the projective limit with respect to $n$ of this exact sequence, we obtain the claim. 
\end{proof}

Recall that we have the isomorphism $\Lie(\cA)(\cO_F)\cong H^0(\cA, \Omega^1_{\cA/\cO_F})^\vee$.
By \autoref{lemma:squarezero}, we have the short exact sequence
\[
0\to  T_p(\Omega) \to A_{\rm inf}^{(1)}\to \cO_{\C_p}\to 0,
\] 
where $T_p(\Omega)$ is an ideal of $A_{\rm inf}^{(1)}$ such that $(T_p(\Omega))^2=0$.

By definition, we have 
\[
\Lie(\cA)(\cO_F)\otimes_{\cO_F}T_p(\Omega) \cong \ker\left(\cA(A_{\rm inf}^{(1)})\to \cA(\cO_{\C_p})\right),
\] 
and hence we have the following short exact sequence of abelian groups with $G_F$-action
\[
0\to \Lie(\cA)(\cO_F)\otimes_{\cO_F}T_p(\Omega)\to \cA(A_{\rm inf}^{(1)})\to \cA(\cO_{\C_p})\to 0.
\]
Consider the following commutative diagram with exact rows
\[
\begin{tikzcd}[row sep = 1.2em]
0 \arrow{r} & \Lie(\cA)(\cO_F)\otimes_{\cO_F}T_p(\Omega)\arrow{r} \arrow{d}{p^n}&\cA(A_{\rm inf}^{(1)}) \arrow{r}{d}\arrow{d}{p^n} &\cA(\cO_{\C_p}) \arrow{r}\arrow{d}{p^n} &0\\
0 \arrow{r} & \Lie(\cA)(\cO_F)\otimes_{\cO_F}T_p(\Omega)\arrow{r} &\cA(A_{\rm inf}^{(1)}) \arrow{r}{d}&\cA(\cO_{\C_p}) \arrow{r} &0. 
\end{tikzcd}
\]
The snake lemma gives a $G_K$-equivariant map 
\[
\nu_n\colon \cA(\cO_{\C_p})[p^n]\cong A(\Kbar)[p^n]\to \Lie(\cA)(\cO_F)\otimes_{\cO_F} \Omega[p^n]
\]
and by taking the projective limit over $n$'s, we obtain a map
\[
\nu\colon T_p(A) \to \Lie(\cA)(\cO_F)\otimes_{\cO_F}T_p(\Omega). 
\]
\begin{proposition}\label{prop:anotherperspective}
The map obtained above
\[
\nu: T_p(A) \to \Lie(\cA)(\cO_F)\otimes_{\cO_F}T_p(\Omega) \subset {\rm Lie}(\cA)(\cO_F)\otimes_{\cO_F}\C_p(1)
\]
coincides with Fontaine's integral, i.e.~we have $\nu=(\varphi_\cA)$.
\end{proposition}

\begin{proof}
In \cite[Section 4, page 394]{wintenberger:comparision}, Wintenberger used a generalization of the above construction to obtain an integration pairing which coincides with the Colmez integration pairing $\langle \cdot ,\cdot \  \rangle_{\Cz}$. 
The result now follows from \cite[Proposition 6.1]{Colmez1992}. 
\end{proof}

 \section{\bf Consequences of \autoref{thm:kernelformal}: ramification of $p$-power torsion points on $\widehat{\cA}$}
 \label{sec:differentperspective}
In this section, we use the interpretation of the Fontaine integral from \autoref{prop:anotherperspective} and \autoref{thm:kernelformal} to deduce properties concerning the ramification of $p$-power torsion points on the formal group of $A$. 

To begin, we recall the diagram  
\[
\begin{tikzcd}[row sep = 1.2em]
0 \arrow{r} & \Lie(\cA)(\cO_F)\otimes_{\cO_F}T_p(\Omega)\arrow{r} \arrow{d}{p^n}&\cA(A_{\rm inf}^{(1)}) \arrow{r}{d}\arrow{d}{p^n} &\cA(\cO_{\C_p}) \arrow{r}\arrow{d}{p^n} &0\\
0 \arrow{r} & \Lie(\cA)(\cO_F)\otimes_{\cO_F}T_p(\Omega)\arrow{r} &\cA(A_{\rm inf}^{(1)}) \arrow{r}{d}&\cA(\cO_{\C_p}) \arrow{r} &0. 
\end{tikzcd}
\]
Above, we only wrote a piece of the snake lemma, and by writing more of it, we have an exact sequence of $G_K$-modules
\[
0\to \cA(A_{\rm inf}^{(1)})[p^n]\to \cA(\cO)[p^n]\to \Lie(\cA)(\cO_F)\otimes_{\cO_F} \Omega[p^n]. 
\]
By taking projective limits, we have the \textit{exact} sequence
\begin{equation}\label{eqn:reducingconj}
0\to T_p(\cA(A_{\rm inf}^{(1)}))\to T_p(A)\stackrel{\varphi_\cA}{\to}\Lie(\cA)(\cO_F)\otimes _{\cO_F} T_p(\Omega)\subset \Lie(\cA)(\cO_F)\otimes_{\cO_F} \C_p(1).
\end{equation}
Therefore, \autoref{thm:kernelformal} implies that $T_p(\widehat{\cA}(A_{\rm inf}^{(1)}))=0$.

To study consequences of this property, we will use another ring instead of $A_{\rm inf}^{(1)}$. 

\begin{definition}[\protect{\cite{fontaine}}]\label{defn:Df}
Let $\theta\colon A_{\rm inf}^{(1)}\to \cO_{\C_p}$ denote the projection map. Then, we define $D_f:=\theta^{-1}(\cO)$. 
In \cite[Remark 1.4.7]{fontaine}, Fontaine gives the following construction of $D_f$. 
Let us recall that
\[
V_p(\Omega)=T_p(\Omega)\otimes_{\Z_p} \Q_p=\varprojlim \left(\Omega\stackrel{p}{\leftarrow} \Omega\stackrel{p}{\leftarrow}\cdots\stackrel{p}{\leftarrow}\Omega \cdots  \right)
\]
and that $\Omega$ and $V_p(\Omega)$ are $\cO$-modules. We make $R:=V_p(\Omega)\oplus \cO$ into a commutative ring by defining multiplication as follows:~$(u, \alpha)(v, \beta)=(\beta u+\alpha v, \alpha\beta)$ for $(u, \alpha), (v, \beta)\in R$, i.e. we require that $V_p(\Omega)$ is an ideal of $R$ of square $0$. Then we have
\[
D_f=\{\left(u=(u_n)_{n\ge 0},\alpha\right)\in R\ |\ d(\alpha)=u_0\}. 
\]
\end{definition}

By \autoref{defn:Df}, we have an exact sequence of $G_K$-modules
\[
0\to T_p(\Omega)\to D_f\stackrel{\theta}{\to} \cO\to 0,
\]
where $\theta(u,\alpha)=\alpha$, and the $p$-adic completion of $D_f$ is $A_{\rm inf}^{(1)}$. 
We note that we may construct the diagram above in the same way using $D_f$ instead of $A_{\rm inf}^{(1)}$, which produces the exact sequence \eqref{eqn:SESDf} with $D_f$ instead of $A_{\rm inf}^{(1)}$.
Instead of the exact sequence \eqref{eqn:reducingconj} above, we will have the following \textit{exact} sequence
\[
0\to T_p(\cA(D_f))\to T_p(A)\stackrel{\varphi_\cA}{\to}\Lie(\cA)(\cO_F)\otimes_{\cO_F} T_p(\Omega)\subset \Lie(\cA)(\cO_F)\otimes_{\cO_F} \C_p(1).
\]
Again, the \autoref{thm:kernelformal} implies that $T_p\bigl(\widehat{\cA}(D_f)\bigr)=0$.  

We will use this observation to deduce that $T_p\bigl(\widehat{\cA}(\cO^{(1)})\bigr)=0$. 
In order to do so, we need to show that $\widehat{\cA}[p^n](D_f)\cong \widehat{\cA}[p^n](\cO^{(1)})$, for all $n\ge 1$, which is accomplished through the following two lemmas.

\begin{lemma}\label{lemma:O^1torsionDftorsion}
Let $x\in \cA[p^n](\cO^{(1)})$, then there is $x'\in \cA(D_f)$ with $\theta(x')=x$ and such that $[p^n](x')=0$.
\end{lemma}

\begin{proof}
Recall that $D_f:=\left\lbrace\bigl((x_n)_n, y)\in V_p(\Omega)\times \cO\ |\ x_0=dy\right\rbrace$, i.e. we have an exact sequence
\[
0\to T_p(\Omega)\to D_f\stackrel{\theta}{\to}\cO\to 0
\]
and this exact sequence splits over $\cO^{(1)}\subset \cO$, i.e. the following diagram is cartesian and has exact rows
\[
\begin{array}{cccccccccc}
0&\longrightarrow &T_p(\Omega)&\longrightarrow& D_f&\stackrel{\theta}{\longrightarrow}&\cO&\longrightarrow& 0\\
&&||&&\cup&&\cup\\
0&\longrightarrow &T_p(\Omega)&\longrightarrow &T_p(\Omega)\oplus \cO^{(1)}&\stackrel{\theta}{\longrightarrow}&\cO^{(1)}&\longrightarrow& 0.
\end{array}
\]
In particular, the section $s\colon \cO^{(1)} \longrightarrow D_f$ is defined by $s(x):=(0,x)$. Then $s$ defines a morphism $s\colon \cA(\cO^{(1)})\to \cA(D_f)$, and if $x\in \cA[p^n](\cO^{(1)})$, then $s(x)\in \cA[p^n](D_f)$. 
\end{proof}

In the next lemma, we show a converse of \autoref{lemma:O^1torsionDftorsion} at least for the formal group $\widehat{\cA}$ of $A$.

\begin{lemma}
\label{prop:dftorsion}
Let $\widehat{\cA}$ denote the formal group of the abelian scheme $\cA$ and fix $n\ge 1$ an integer.
Let 
\[
0\neq P\in \widehat{\cA}(\cO)[p^n]\backslash \widehat{\cA}(\cO^{(1)})
\]
and let $Q\in \widehat{\cA}(D_f)$ be a point such that $\theta(Q)=P$. Then $[p^m]Q\neq 0$ for all $m\ge n$. 
\end{lemma}

\begin{proof}
Let $m\ge n$ and denote by $$[p^m](X_1,\dots,X_g):=\left(f_1(X_1,\dots,X_g), f_2(X_1,\dots,X_g),\dots, f_g(X_1,\dots,X_g)\right)$$ the multiplication by $p^m$ on $\widehat{\cA}$.
Let  $S:=\cO_F[[X_1,\dots,X_g)]]/I$, where $I$ is the ideal generated by $f_1, f_2,\dots,f_g$, then we know $S$ is a finite flat $\cO_F$-algebra,  in particular $S$ is a free $\cO_F$-module and $\widehat{\cA}[p^m]:={\rm Spec}(S)$ with the co-multiplication of $\widehat{\cA}$, is a finite flat group-scheme, so $\widehat{\cA}\times_{\cO_F}{\rm Spec}(F)$ is an \'etale, therefore smooth, group-scheme over $F$. This implies that the image in $S\otimes_{\cO_F}F$ of the determinant of the matrix:
$$
\left( \begin{array}{cccccc} \frac{\partial(f_1)}{\partial(X_1)} & \frac{\partial(f_1)}{\partial(X_2)}&\dots&\frac{\partial(f_1)}{\partial(X_g)} \\
\vdots&\vdots&\ddots&\vdots\\
\frac{\partial(f_g)}{\partial(X_1)} & \frac{\partial(f_g)}{\partial(X_2)}&\dots&\frac{\partial(f_g)}{\partial(X_g)}\end{array}\right)
$$ 
is a unit.

Let now $P=(x_1,x_2,\dots,x_g)\in \mathfrak{m}_{\cO}^g\in \widehat{\cA}[p^n](\cO)\backslash \widehat{\cA}[p^n](\cO^{(1)})$, i.e. there is $1\le i\le g$ such that 
$x_i$ is not in $\cO^{(1)}$. Let $P'=(y_1,y_2,\dots,y_g)\in \widehat{\cA}(D_f)$ such that $\theta(P')=P$, i.e. $y_j=\alpha_j+x_j$, with 
$\alpha_j\in V_p(\Omega)$ and $d(x_j)=\alpha_{j,0}$, for all $1\le j\le g$. By the above assumption $\alpha_i\neq 0$.
As in $D_f$ we have $\alpha_j\alpha_k=0$ for all $1\le j,k\le g$, the Taylor formula implies that if $[p^m](P')=0$ we must have:
\[
  f_s(x_1,\dots,x_g)+ \sum_{j=1}^g\frac{\partial(f_s)}{\partial(X_j)}(x_1,\dots,x_g)\alpha_j=\sum_{j=1}^g\frac{\partial(f_s)}{\partial(X_j)}(x_1,\dots,x_g)\alpha_j =0
\]
for every $1\le s\le g$.
But the determinant of the matrix
\[
\left( \begin{array}{cccccc} \frac{\partial(f_1)}{\partial(X_1)}(x_1,\dots,x_g) & \frac{\partial(f_1)}{\partial(X_2)}(x_1,\dots,x_g)&\dots&\frac{\partial(f_1)}{\partial(X_g)}(x_1,\dots,x_g) \\
\vdots&\vdots&\ddots&\vdots\\
\frac{\partial(f_g)}{\partial(X_1)}(x_1,\dots,x_g) & \frac{\partial(f_g)}{\partial(X_2)}(x_1,\dots,x_g)&\dots&\frac{\partial(f_g)}{\partial(X_g)}(x_1,\dots,x_g)\end{array}\right)
\]
is a unit in $\overline{K}$, i.e. it is non-zero and $\alpha_i\neq 0$. This is a contradiction. 
\end{proof} 

\begin{remark}
We note that the group-scheme $\widehat{\cA}[p^m]$ is not smooth over $\cO_F$ (for example $S/pS$ could have nilpotents).
We also remark that as $\widehat{\cA}[p^m]\times_{\cO_F}{\rm Spec}(F)$ is smooth, the map $\theta\otimes 1\colon \widehat{\cA}[p^m](D_f\otimes_{\cO_F}K)\to \widehat{\cA}[p^m](\overline{F})$ is surjective, but this is clear as $D_f\otimes_{\cO_F}F=V_p(\Omega)\oplus \overline{F}$.
\end{remark}

\autoref{lemma:O^1torsionDftorsion} and \autoref{prop:dftorsion} imply that the map $\theta$ gives an isomorphism $\widehat{\cA}[p^n](D_f)\cong 
\widehat{\cA}[p^n](\cO^{(1)})$, for all $n\ge 1$.
Combining this with the fact that $T_p\bigl(\widehat{\cA}(D_f)\bigr)=0$, we have the following result.

\begin{theorem}[$=$\autoref{xthm:main1}]
\label{lemma:ram}
Let $A$ be an abelian variety over $F$ with good reduction. Then there is $n_0\ge 1$ such that 
for every $m\ge n_0$ and $0\neq P\in \widehat{\cA}[p^m](\cO)\setminus \widehat{\cA}[p^{n_0 -1 }](\cO)$, we have $P\notin\widehat{\cA}(\cO^{(1)})$
\end{theorem}

\begin{remark} \label{remark:ram}
We observe that the above is a result regarding ramification properties of the $p
$-power torsion points of the formal group of our abelian  variety with good reduction over $F$ 
(c.f.~\autoref{prop:deeplyramified}). More precisely, let $m\ge n_0$ and $0\neq P\in 
\widehat{\cA}[p^m](\cO)\setminus \widehat{\cA}[p^{n_0 -1 }](\cO)$ be as in the theorem above. 
Let $P=(x_1,x_2,...,x_g)$ with $x_i\in \cO$ for $1\le i\le g$.  \autoref{lemma:ram} says the following: let 
$L=K[P]:=K[x_1,x_2,...,x_g]$, let $\pi$ denote a uniformizer of $L$ and let $\cD_{L/K}$ denote 
the different ideal of $L/K$. For every $1\le i\le g$ let $f_i(X)\in \cO_K[X]$ 
be polynomials such that $f_i(\pi)=x_i$, for every $i$. Then there is $1\le j\le g$ such that 
$v(f_j'(\pi))< v(\cD_{L/K})$ (i.e. $x_j\notin \cO^{(1)}$).
\end{remark}


\section{\bf A theorem on the ramification type of the field obtained by adjoining a $p$-torsion point of a formal group}
\label{sec:ramification}
In this subsection we study the ramification properties of the extension $K[P]/K$, where $P$ is a non-zero $p$-power torsion point of the formal group of $A$.

We continue to denote by $K$ the completion of the maximal unramified extension of $\Q_p$ in an algebraic closure of $\Q_p$, which we denote $\Kbar$. 
Let $\sF$ denote a formal group of dimension $g$ over ${\rm Spf}(\cO_K)$. 
For example,  $\sF$ can be the formal group of the N\'eron model of $A$ over ${\rm Spf}(\cO_K)$.

To begin, we define the notion of a strict formal group. 

\begin{definition}\label{defn:strict}
Consider the multiplication-by-$p$ map \[[p](X_1,\dots.,X_g) = (f_1(X_1,\dots.,X_g),f_2(X_1,\dots.,X_g),\dots,f_g(X_1,\dots.,X_g))\] on $\sF$ where each $f_i(X_1,\dots.,X_g)$ is a power series in with coefficients in $\cO_K$. 
For each $1\leq i\leq g$, define $F_i(X_1,\dots,X_g)$ to be the form comprised of monomials of $f_i$ which have unit coefficient and minimal degree, where we consider each monomial $X_1,\dots,X_g$ to be of degree 1.

Let $d_1,\dots,d_g$ denote the degree of these forms $F_1(X_1,\dots,X_g),\dots,F_g(X_1,\dots,X_g)$, respectively, which we note are (possibly distinct) powers of $p$. 
Let $G_1(X_1,\dots,X_g),\dots,G_g(X_1,\dots,X_g)$ denote the reductions modulo $p$ of the forms $F_1(X_1,\dots,X_g),\dots,F_g(X_1,\dots,X_g)$.

Consider the system of equations
\begin{equation}\label{eqn:strict}
G_1(X_1,\dots,X_g) = G_2(X_1,\dots,X_g) = \cdots = G_g(X_1,\dots,X_g) = 0,
\end{equation}
We say that the formal group $\sF$ is \cdef{strict} if $d_1 = d_2 = \cdots = d_g$ and the only solution to \eqref{eqn:strict} is $(0,0,\dots,0) \in (\overline{\mF_p})^g$. 
\end{definition}

\begin{remark}\label{rem:strict1}
If $\sF$ is a formal group of dimension $1$, then it is clear that $\sF$ is strict since we will have that $F_1(X_1) = uX_1^{p^h}$, where $h$ is the height of $\sF$. 
Moreover, if $\sF$  is the product of $1$-dimensional formal groups, then again $\sF$ is strict. 
\end{remark}

\begin{remark}\label{rem:strict2}
We can given an equivalent characterization of strict as follows. 
Consider the $g\times g$ matrix $M = (a_{ij})$ where the entry $a_{ij}$ consists of the coefficient of $X_i$  in the linear form $G_j$ for each $1\leq i,j \leq g$. 
The condition that $\sF$ be strict is equivalent to the determinant of $M$ being non-zero.

We refer the reader to \cite[Remark 4.14]{arias} for an example of a $2$-dimensional formal group of height $4$ where this condition holds, and here we note that the above degrees all equal $p^2$. Moreover, we remark that the proof from \textit{loc.~cit.~}holds for any $g$-dimensional formal group of height $2g$ where the degrees $d_1 = d_2 = \cdots = d_g$ all equal $p^2$. 
\end{remark}

With this definition, we can state our main result.

\begin{theorem}\label{thm:tamelyramifiedtorsion}
Let $\sF$ be a strict formal group of dimension $g$.
For $0\neq P = (x_1,\dots,x_g) \in \sF[p](\cO)$, the field of definition $K(P)/K$ is tamely ramified and $\cO_{K(P)} \cong \cO_K[x_1,\dots,x_g]$. 
Moreover, $K(\sF[p])/K$ is tamely ramified. 
\end{theorem}

\begin{proposition}
\label{prop:torsionram}
Let $\sF$ be a strict the formal group of dimension $g$.
For every $0\neq P\in\sF[p](\cO)$, the coordinates of $P$ are not all in $\cO^{(1)}$. 
\end{proposition}

\begin{proof}
Let $0\neq P = (x_1,\dots,x_g) \in \sF[p](\cO)$ be a non-zero $p$-torsion point. 
By \autoref{thm:tamelyramifiedtorsion}, we know that the extension $K(P)/K$ is tamely ramified and that there exists some coordinate $x_i$ which is a uniformizer for $K(P)/K$. 
By \cite[Proposition III.6.13]{Serre}, we have that $v(\mathcal{D}_{K(P)/K}) > 0 $ where $\mathcal{D}_{K(P)/K}$ is the different ideal of $K(P)/K$. 
Now since $x_i$ is a uniformizer for $K(P)/K$, we have that 
\[
\delta(x_i) = -v(\mathcal{D}_{K(P)/K})  < 0,
\]
where $\delta$ is the function defined in \autoref{defn:delta}. 
By \autoref{lemma:propertiesofdelta}.(5), $x_i \notin \cO^{(1)}$ as desired.  
\end{proof}

For the remainder of this section, we focus on proving \autoref{thm:tamelyramifiedtorsion}. The proof can be broken down into three steps.
\begin{enumerate}
\item Given a non-zero $p$-torsion point $P = (x_1,\dots,x_g) \in \sF[p](\cO)$, we will carefully construct linear combinations $z_i^*$ of the $x_1,\dots,x_n$ which satisfy nice properties in terms of their valuations and distances between their $\overline{K}$-conjugates. See \autoref{lemma:niceisom}. 
\item Next, we consider the change of variables (i.e., the isomorphism of formal groups) which sends the coordinate $X_i$ to the linear combination $Z_i^*$ described above. We use the properties of the $z_i^*$ and the strictness of $\sF$ to precisely determine the valuation of $z_i^*$ and to estimate the valuation of the difference between them and their $\overline{K}$-conjugates. 
\item Finally, we use Krasner's lemma to deduce that one of the original coordinates $x_1,\dots,x_g$ of $P$ must be a uniformizer for the maximal order of $K(P)$, from which \autoref{thm:tamelyramifiedtorsion} follows. 
\end{enumerate}

\begin{lemma}\label{lemma:niceisom}
Let $\sF$ be a formal group of dimension $g$ over $\Spf(\cO_K)$. 
Let $0\neq P= (x_1,\dots,x_g) \in \sF[p](\cO)$. 
There exist linear combinations $z_1^*,\dots,z_g^*$ of $x_1,\dots,x_g$ with coefficients in $(\cO_K)^{\times} \cup \{0\}$ which satisfy: 
\begin{enumerate}
\item $K(z_{i}^*) \cong K(P)$,
\item $v(z_{i}^*) =  \min\{ v(x_{1}),\dots, v(x_{g})\}$,
\item $v(z_{i}^* - \sigma(z_{i}^*)) = \min\{ v(x_{1} - \sigma(x_{1})),\dots, v(x_{g} - \sigma(x_{g}))\}$ for all $\sigma \in \gal(\widetilde{K(P)}/K)$ where $\widetilde{K(P)}$ is the Galois closure of $K(P)$,
\end{enumerate}
and such the matrix $M$ representing the change of coordinates $(z_1^*,\dots,z_g^*)^t= M (x_1,\dots,x_g)^t$ is invertible. Here 
the exponent $t$ indicates the transpose of a matrix.
\end{lemma}

\begin{proof}
Let $e := [K(P) : K]$. 
Our proof will involve making a series of linear combinations. 
To begin, we will construct the element $z_1^*$. 
First, consider all the linear combinations of the form
\begin{equation}\label{eqn:1s}
\cB_1:=\{z = u_1x_1 + \cdots + u_gx_g \text{ where } u_i\in(\cO_K)^{\times} \cup \{0\} \text{ and }u_1 \neq 0 \}. 
\end{equation}

By our assumptions on $K$, the set of $u_i (\mbox{mod }p)$ is infinite, with $u_i$ as in the above formula, and hence we may find one linear combination, call it $z_1$, satisfying the following two conditions:
\begin{enumerate}
\item $v(z) \geq v(z_1)$ for all other linear combinations $z$ from $\cB_1$,
\item $K(z_1) \cong K(P)$.
\end{enumerate}
To show that condition (2) holds, consider the following. 
There are exactly $e$ embeddings of $K(x_1, \dots ,x_g)$ into the fixed algebraic closure $\Kbar$, call them $\sigma_1, \dots, \sigma_e$. Note that the vectors ${\textbf{w}}_j  := (\sigma_j(x_1), \dots, \sigma_j(x_g))$, $ 1\le j \le e$, are distinct. Indeed, if for some $ i \ne j$ the vectors ${\textbf{w}}_i$ and ${\textbf{w}}_j$ coincide, then $\sigma_i$ and $\sigma_j$ will coincide at $x_1,\dots,x_g$ and so they will coincide on $K(x_1,\dots,x_g)$, which is not the case. 

Consider now, for each pair $(i,j)$ with $1\le i, j \le e$ and $i\ne j$, the hyperplane ${\mathcal H}_{i,j}$ given by 
\[
{\mathcal H}_{i,j} = \brk{ (c_1,c_2,\dots,c_g) :  c_1, \dots, c_g \in K , \sum_{l = 2}^g  c_l ( \sigma_i(x_l) - \sigma_j(x_l)) = 0 }.
\]
Since the vectors ${\textbf{w}}_j$, $1 \le j \le e$ are distinct, none of ${\mathcal H}_{i,j}$ covers the full space. 
Denote by ${\mathcal H}$ the union of these finitely many hyperplanes.
Choose now any 
$c_1,\dots,c_g \in K$ such that the point $(c_1,c_2,\dots,c_g)$ lies outside ${\mathcal H}$. Then we claim that the element
\[
z := c_1x_1+ \cdots + c_g x_g
\]
satisfies $K(z) = K(x_1,\dots,x_g)$. Indeed, $\sigma_1(z),\dots,\sigma_e(z)$ are distinct. For, if two of them are equal, say 
$\sigma_i(z) = \sigma_j(z)$ with $i\ne j$, then $(c_1,c_2\dots,c_g) $ is forced to lie in ${\mathcal H}_{i,j}$. Thus
$\sigma_1(z),\dots,\sigma_e(z)$ are distinct, so $z$ has at least $e$ distinct conjugates over $K$. Hence $[K(z):K] \ge e$ and in conclusion $K(z) = K(x_1,\dots,x_g)$ i.e., $z_1$ has degree $e$ over $K$.

We pause to note that the matrix $M$ representing the change of coordinates $(z_1,x_2,\dots,x_g)^t= M (x_1,x_2,\dots,x_g)^t$ is invertible.
Indeed, the matrix $M$ has units along the diagonal, the coefficients of the linear combination $z_1$ in the first row, and zeros elsewhere, hence the determinant is a unit. 

We now look at the distances between these linear combinations and various of their conjugates over $K$. 
Fix $\sigma \in \text{Gal}(\overline{\Q_p}/K)$ and consider the infimum of the values $v(z - \sigma(z))$ for the linear combinations $z$ as in \eqref{eqn:1s}, i.e. we look at $\inf\{ v\left(z-\sigma(z)\right)\ |\ z\in \cB_1\}$.  
We first show that this infimum exists and is attained by some linear combination. 
Note that if $\sigma_{|K(z_1)} = \text{id}$, then since all linear combinations belong to $K(x_1,\dots,x_g)$, we have that $z - \sigma(z) = 0$ for all linear combinations $z$, i.e. $\inf\{ v\left(z-\sigma(z)\right)\ |\ z\in \cB_1\}=\infty=v(z_1-\sigma(z_1))$. 
If we consider $\sigma$ such that $\sigma_{|K(z_1)} \neq \text{id}$, then we at least have that $v(z - \sigma(z)) < \infty$ for some linear combinations $z$, for example for $z = z_1$.
We have that the set of possible valuations is discrete, and the set of values $v(z - \sigma(z))$ has a lower bound, namely $\min\{v(x_1 - \sigma(x_1)),\dots,v(x_n - \sigma(x_g))\}$. 
Thus, the infimum of the set of values $v(z - \sigma(z))$ is attained by some linear combination $z$ from \eqref{eqn:1s}, but note it need not be obtained by $z_1$. 

Let $G$ denote $\gal(\widetilde{K(P)}/K)$ where $\widetilde{K(P)}$ is the Galois closure of $K(P)/K$.
For a fixed  $\sigma\in G$, let $z_{\sigma}\in \cB_1$ denote one such linear combination attaining the minimum $v(z_\sigma - \sigma(z_\sigma))=\min\{ v\left(z-\sigma(z)\right)\ |\ z\in \cB_1\}$ . 
We claim that we can find a linear combination of $z_1$ and of all of these $z_{\sigma}$ where $\sigma \in G$ which simultaneously achieves these minima. 
To do this, consider all linear combinations
\begin{equation}\label{eqn:2}
z^* = z_1 + \sum_{\sigma \in G}u_{\sigma}z_{\sigma} \quad  \text{ where } u_{\sigma} \in (\cO_K)^{\times}\cup \{0\}.
\end{equation}
We will choose the $u_\sigma$'s such that each such $z^*$ will live in $\cB_1$. 
 
To achieve our desired simultaneous minima, we start with one $\sigma \in G$, call it $\sigma_1$. 
First, we let $z^*= z_1$. 
This linear combination might work in that it already attains the minimum at $\sigma_1$; by this we mean that $v(z - \sigma_1(z)) \ge v(y_1 - \sigma_1(y_1))$ holds for all linear combinations $z\in \cB_1$ from \eqref{eqn:1s}. 
If this is the case, then we set $y_1:=z_1$. 
Now suppose that $z_1$ does not attain the minimum at $\sigma$. 
In this case, we may use any unit $u\in (\cO_K)^{\times}$ and set $z^*:=  z_1 + uz_{\sigma_1}$. 
Indeed, for any unit $u \in (\cO_K)^{\times}$ and $z^*$ above, we have that 
$$v(z^* - \sigma_1(z^*)) =v\left(z_1-\sigma_1(z_1)+u(z_{\sigma_1}-\sigma_1(z_{\sigma_1})\right)=
 v(z_{\sigma_1} - \sigma_1(z_{\sigma_1})),
 $$ 
 because $v(z_1-\sigma_1(z_1))>v(z_{\sigma_1}-\sigma_1(z_{\sigma_1})$.
Let $y_1:=z_1+uz_{\sigma_1}$ with unit $u \in (\cO_K)^{\times}$ such that $y_1\in \cB_1$. We have that for such
 $y_1\in \cB_1$, $v(y_1-\sigma_1(y_1))\le v(z-\sigma_1(z)$ for all $z\in \cB_1$ and that the $u's$ with $y_1\in \cB_1$
 have the property that $u \pmod p$ avoids a finite number of elements in $\overline{\F_p}$. 
 
For another automorphism $\sigma_2\in G$,  we proceed along the same lines, that is: if $y_1$ has the property
that $v(y_1-\sigma_2(y_1))\le v(z-\sigma_2(y_1))$ for all $z\in \cB_1$ we set $y_2:=y_1$.
If the above is not true, let $z^*:=y_1+ux_{\sigma_2}$, for some $u\in \cO_K^\times$. Then as above we have:
$v(z^*-\sigma_2(z^*))=v\left(x_{\sigma_2}-\sigma_2(x_{\sigma_2})\right)$ therefore $z^*$ realizes the minimum for $\sigma_2$, for all
$u's$ for which $z^*\in \cB_1$. What about for $\sigma_1$. The worst that can happen is that $v(y_1-\sigma_1(y_1))=v\left(x_{\sigma_2}-\sigma_1(x_{\sigma_2})\right)$, i.e. ~if we denote by $\pi$ a uniformizer of $K(P)$, we have $y_1-\sigma_1(y_1)=a\pi^\alpha$
$x_{\sigma_2}-\sigma_1(x_{\sigma_2})=b\pi^\alpha$, with $a,b\in \cO_{K(P)}^\times$. Therefore $z^*-\sigma_1(z^*)=(a+ub)\pi^a.$
Now the residue field of $\cO_{K(P)}$ is $k$, therefore by choosing $u\in \cO_K^{\times}$ such that
$a+ub(\mbox{ mod }\pi)\neq 0$ we have $v(z^*-\sigma_1(z^*))=v(y_1-\sigma_1(y_1))$ and therefore $y_2:=z^*$ realizes the minima for 
both $\sigma_1$ and $\sigma_2$.

Continuing in this fashion, we arrive at the conclusion that there exist linear combinations of the form
\begin{equation}
z^*_1 = z_1 + \sum_{\sigma \in G}u_{\sigma}z_{\sigma}  
\end{equation}
where $u_{\sigma} \in (\cO_K)^{\times}\cup \{0\}$ which satisfy the following four conditions:
\begin{enumerate}
\item $z_{1}^* \in \cB_1$, 
\item $K(z^*_1) \cong K(P)$,
\item $v(z^*_1) = \min\{ v(x_1),\dots, v(x_n)\}$,
\item $v(z^{*}_1 - \sigma(z^*_1)) = \min\{ v(x_1 - \sigma(x_1)),\dots, v(x_n - \sigma(x_g))\}$ for all $\sigma \in G$,
\end{enumerate}
as desired. 
Again, we pause to note that the matrix $M$ representing the change of coordinates $(z_1^*,x_2,\dots,x_g)^t= M (x_1,x_2,\dots,x_g)^t$ is invertible.
Indeed, the matrix $M$ has units along the diagonal, the coefficients of the linear combination $z_1^*$ in the first row, and zeros elsewhere, hence the determinant is clearly a unit. 

We now wish to iterate this construction as follows. 
First, consider the set of all linear combinations
\begin{equation}\label{eqn:firstlinearcombo}
\cB_2:=\{z' = u_1z_1^* + u_2x_2 + \cdots + u_gx_g \text{ where } u_i\in(\cO_K)^{\times} \cup \{0\} \text{ and }u_2 \neq 0 \}. 
\end{equation}
Then, we can repeat the above construction to arrive at a linear combination $z_2 \in \cB_2$ satisfing:
\begin{enumerate}
\item $v(z') \geq v(z_2)$ for all other linear combinations $z$ from $\cB_2$,
\item $K(z_2) \cong K(P)$.
\end{enumerate}
Furthermore, we can follow the above construction to say that there exist linear combinations of the form
\begin{equation}
z^*_2 = z_2 + \sum_{\sigma \in G}u_{\sigma}z_{\sigma}  
\end{equation}
where $u_{\sigma} \in (\cO_K)^{\times}\cup \{0\}$ which satisfy the following four conditions:
\begin{enumerate}
\item $z_{2}^* \in \cB_2$, 
\item $K(z^*_2) \cong K(P)$,
\item $v(z^*_2) = \min\{ v(x_1),\dots, v(x_n)\}$,
\item $v(z^{*}_2 - \sigma(z^*_2)) = \min\{ v(x_1 - \sigma(x_1)),\dots, v(x_g - \sigma(x_g))\}$ for all $\sigma \in G$,
\end{enumerate}
Note that the matrix $M'$ representing the change of coordinates $(z_1^*,z_2^*,\dots,x_g)^t = M'(z_1^*,x_2,\dots,x_g)^t$ is invertible. 
Indeed, $M'$ has units on the diagonal, the coefficients of the linear combination $z_2^*$ in the second row, and has zeros everywhere else. 
Although this matrix is not triangular, we can make it so by switching the second row with the first and interchanging the first and second columns; these operations will not change the determinant. After these operations, the matrix becomes triangular with units on the diagonal, and hence the will be invertible.
Moreover, we see that the matrix $M''$ representing the change of coordinates $(z_1^*,z_2^*,\dots,x_g)^t = M''(x_1,x_2,\dots,x_g)^t$ is invertible since $M'' = M' \cdot M$. 

We continue in this fashion for all of the remaining coordinates $x_3,\dots,x_g$ and arrive at our desired claim, namely that there exists linear combinations $z_1^*,\dots,z_g^*$ of $x_1,\dots,x_g$ with coefficients in $(\cO_K)^{\times} \cup \{0\}$ which satisfy: 
\begin{enumerate}
\item $K(z_{i}^*) \cong K(P)$,
\item $v(z_{i}^*) =  \min\{ v(x_{1}),\dots, v(x_{g})\}$,
\item $v(z_{i}^* - \sigma(z_{i}^*)) = \min\{ v(x_{1} - \sigma(x_{1})),\dots, v(x_{g} - \sigma(x_{g}))\}$ for all $\sigma G$
\end{enumerate}
and such the matrix $M$ representing the change of coordinates $(z_1^*,\dots,z_g^*)^t= M (x_1,\dots,x_g)^t$ is invertible.
\end{proof}

We now complete the proof of  \autoref{thm:tamelyramifiedtorsion}. 
\begin{proof}[Proof of \autoref{thm:tamelyramifiedtorsion}]
Fix $0\neq P = (x_1,\dots,x_g) \in \sF[p](\cO)$ and let $e = [K(P) : K]$. 
We first remark that since $K$ is completion of the maximal unramified extension, we have that $e > 1$, and hence the extension $[K(P) : K]$ is totally ramified. Indeed, this follows from that fact that the group-scheme $\sF[p]$ is connected. 
In \autoref{lemma:niceisom}, we constructed linear combinations $z_i^*$ of $x_1,\dots,x_g$ satisfying certain properties. 
Let $Z_i^*$ denote the same linear combinations of the coordinates $X_1,\dots,X_g$ (so if we were to evaluate $Z_i^*$ at $(x_1,\dots,x_g)$ we would recover $z_i^*$). 
The last condition from \autoref{lemma:niceisom} implies that the change of variables $(X_1,\dots,X_g) \mapsto (Z_1^*,\dots,Z_g^*)$ is an isomorphism of formal groups.

We claim that the isomorphism of formal groups $(X_1,\dots,X_g) \mapsto (Z_1^*,\dots,Z_g^*)$ will preserve strictness. 
As each of the $Z_i^*$ are linear combinations of $X_1,\dots,X_g$ with coefficients in $(\cO_K)^{\times} \cup \{0\}$, this isomorphism of formal groups will act linearly on terms of minimal degree, and hence it changes $F_1,\dots,F_g$ by a linear transformation, which is invertible. 
We note that it also does the same to $G_1,\dots,G_g$, therefore, it transforms the set of solutions of the system by an invertible transformation. Moreover, the system having or not having a single solution $(0,\dots,0)$ is the same before or after an isomorphism.
We pause to note that it is crucial that the degrees $d_1, d_2, \dots, d_g$ from \autoref{defn:strict} are all equal

For the remainder of the proof, we work with this isomorphic formal group with coordinates $(Z_1^*,\dots,Z_g^*)$. 
We note that the vector $(z_1^*,\dots,z_g^*)$ will reduce mod $p$ to the point $(0,\dots,0) \in k^g$ because all $z_i^*$ have valuations strictly positive. 
But the $z_i^*$ have the same valuation so we can divide all of them by one of them, and consider the vector $(z_1^*/z_g^*,\dots,z_{g-1}^*/z_g^*,1)$, which will not reduce the zero vector over the residue field. 
Since $\sF$ was assumed to be strict, the reduction of $(z_1^*/z_g^*,\dots,z_{g-1}^*/z_g^*,1)$ cannot be a common root of all of $G_1,\dots,G_g$. Therefore, there exists an index $j$ for which $G_j(z_1^*/z_g^*,\dots,1)$ is not zero in the residue field $k$, and hence the valuation of  $F_j(z_1^*,\dots,z_n^*)$ equals the valuation of each of its individual monomials.

We now want to determine the valuation of $z_j^*$, and hence the valuation of every other $z_i^*$ as they have the same valuation. 
By considering the equation $f_j(z_1^*,\dots,z_g^*)$, we have 
\[
0 = pz_j^* + p(\text{terms of degree between $2$ and $d_j - 1$}) + F_j(z_1^*,\dots,z_n^*) + (\text{higher degree terms}).
\]

We claim that $v(z_j^*) = 1/(d_j-1)$ where $d_j$ is the degree of $F_j$.
To see this choose a unit $u_1$ in $\cO_K$ that is a representative for the element in the residue field corresponding to $z_1^*/z_g^*$, and similarly choose $u_2,\dots,u_{g-1}$. Then $F_j(u_1,\dots,u_{g-1},1)$ is a unit in $\cO_K$, because its image in the residue field is nonzero, by the above choice of $j$. 
But this is a form (of degree $d_j$) so we can divide by $u_j$ inside $F_j$, and re-denoting the $u_j$'s in consideration, we have that $F_j( u_1,\dots,u_{j-1}, 1, u_{j+1} ,\dots,u_g)$ is a unit, and moreover, 
\[
F_j(z_1^*,\dots,z_g^*)  =  F_j(u_1,...1,...u_g) z_j^{*^{d_j}}  +  ( \text{terms of strictly larger valuation}).
\]
Plugging this into the the above equation, we arrive at the equation
\begin{equation}\label{eqn:firstpowerseries}
0 = pz_j^* + F_j(u_1,\dots,1,\dots,u_g) z_j^{*^{d_j}}  + (\text{terms of strictly larger valuation}).
\end{equation}
The minimum valuation in the equality of \eqref{eqn:firstpowerseries} must be attained in at least two terms, and these terms are forced to be $pz_j^*$ and $F_j(u_1,\dots,1,\dots,u_g) z_j^{*^{d_j}}$. Our claim now follows since $F_j(u_1,\dots,1,\dots,u_g) $ is a unit in $\cO_K$.

We now want to study the relationship between the valuations $v(z_j^* - \sigma(z_j^*))$ where $\sigma \in G$. 
Let $u := F_j(u_1,\dots,1,...u_g)$ which is a unit in $\cO_K$. 
For each $\sigma \in G$, we will consider \eqref{eqn:firstpowerseries} and 
\begin{equation}\label{eqn:secondpowerseries}
0 = p\sigma(z_j^*) + u \sigma(z_j^{*})^{d_j} + (***)
\end{equation}
where $(***)$ corresponds to terms strictly larger valuation. 
If we subtract equality \eqref{eqn:secondpowerseries} from \eqref{eqn:firstpowerseries}, we arrive at the following:
\begin{equation}\label{eqn:finalpoweseries}
0 = p(z_j^* - \sigma(z_j^*)) + u(z_j^{*^{d_j}} - \sigma(z_j^{*})^{d_j}) + (\text{interesting terms}).
\end{equation}
By condition (3) of \autoref{lemma:niceisom}, the valuation of the ``interesting terms" from \eqref{eqn:finalpoweseries} will be larger than $v(z_j^* - \sigma(z_j^*))$. 
Indeed, these interesting terms are in fact monomials of the form $z_1^{*^{m_1}} z_2^{*^{m_2}} \dots z_n^{*^{m_g}}$ where some of the $m_i$ could be zero and the total degree is strictly greater than $d_j$. 
In any case, we may deal with the difference of such monomials and their conjugates as follows. Suppose for example, that we have a term of the for $z_2^{*^{m_2}}z_3^{*^{m_3}} - \sigma ( z_2^{*^{m_2}}z_3^{*^{m_3}})$. Then by adding and subtracting the term  $z_2^{*^{m_3}} \sigma( z_3^{*^{m_3}})$, the difference we need to deal with can then be written as $(z_2^* - \sigma(z_2^*))$ times something of positive valuation, plus $(z_3^* - \sigma(z_3^*)$ times something of positive valuation, and so we can use property (3) of \autoref{lemma:niceisom} for this particular $\sigma$ to get our desired claim. 

We now arrive at the crucial claim of the proof. 
Recall that $G$ denotes the Galois group of the Galois closure of $K(P)/K$. 
We claim that $v(z_j^* - \sigma(z_j^*)) = v(z_j^*)$ for all $\sigma \in G$. 
Assume that $v(z_j^* - \sigma(z_j^*)) = v(z_j^*) + t$  where $t > 0$. 
By the above discussion and condition (3) of \autoref{lemma:niceisom}, we can save the value $t$ from each term, and since there must be at least two terms of equal valuation in \eqref{eqn:finalpoweseries}, we have that
\[
v(p(z_j^* - \sigma(z_j^*))) \geq v(u(z_j^{*^{d_j}} - \sigma(z_j^{*})^{d_j})).
\]
Note that we cannot guarantee the equality of these valuations because there could be other terms of total minimal degree other than $u(z_j^{*^{d_j}} - \sigma(z_j^{*})^{d_j})$, but the above inequality will suffice. 
Using the above inequality and previous equality $v(pz_j^*) = v(uz_j^{*^{d_j}})$, and letting $w = \sigma(z_j^*)/z_j^*$, we have that 
\[
t = v(z_j^* - \sigma(z_j^*))  - v(z_j^*) = v((z_j^* - \sigma(z_j^*))/z_j^*) = v(1 - w) \geq v(1 - w^{d_j}).
\]
Let $y = 1 - w$. We have that $v(y) = t$, which by assumption is strictly greater than 0. 
We now arrive at a contradiction by considering the above inequality and the equation
\[
1 - w^{d_j} = 1 - (1 - y)^{d_j} = d_jy + (\text{terms times }y^2) + y^{d_j},
\]
and noting that all terms in the above have valuation strictly greater than $v(y) = t$. 
Therefore, we have that $t= 0 $, and hence $v(z_j^* - \sigma(z_j^*)) = v(z_j^*)$ for all $\sigma \in G$. 

To conclude our proof, we use Krasner's lemma \cite[Exercise II.2.1]{Serre} to explicitly describe the extension $K(z_j^*)/K$ and show that $[K(z_j^*) : K] = d_j - 1$. 
Recall that our $z_j^*$ satisfies the equation \eqref{eqn:firstpowerseries}. 
Consider the polynomial $P(Z) = p + uZ^{d_j - 1}$ where $u = F_j(u_1,...1,...u_n)$ as above. 
Note that $z_j^*$ is not a root of $P(Z)$, but it satisfies the following inequality $v(P(z_j^*)) > v(pz_j^*) $ where the right side here is also equal to $v( u z_j^{*^{d_j}})$.

On the other hand, the roots $\theta_1,\dots,\theta_{d_j - 1}$ of $P(Z)$ each have valuation exactly $1/(d_j - 1)$ since $P(Z)$ is Eisenstein at $p$. 
We now compute the valuation of the derivative of $P(Z)$ evaluated at a root in two different ways. 
First, $P'(\theta_l) = u \prod_{i\neq l} (\theta_l - \theta_i)$ and hence
\[
v(P'(\theta_l)) = \sum_{i\neq l} v(\theta_l - \theta_i) \geq \frac{d_j - 2}{d_j - 1}.
\]
Second, we directly compute that $P'(\theta_l) = (d_j - 1)u\theta_l^{d_j - 2},$ which yields
\[
v(P'(\theta_l)) = \frac{d_j - 2}{d_j - 1}. 
\]
The first inequality and the second equality imply that $v(\theta_l - \theta_i) = 1/(d_j - 1)$ for all $1 \leq i \neq j \leq d_j - 1$. 
We also have that $P(z_j^*) = u(z_j^* - \theta_1)\cdots (z_j^* - \theta_{d_j - 1})$ and hence 
\[
v(P(z_j^*)) = \sum_{i = 1}^{d_j -1} v(z_j^* - \theta_i).
\]
There are exactly $d_j - 1$ terms here and since $v(P(z^*)) > 1$, it follows that at least one term must be strictly larger than $1/(d_j - 1)$. 
Without loss of generality, we may assume that $v(z_j^* - \theta_1) > 1/(d_j - 1).$ 
Now we have the strict inequality
\[
v(z_j^* - \theta_1) > 1/(d_j - 1) = v(\theta_1 - \theta_i)
\]
for all $1 < i \leq d_j - 1$, and hence Krasner's lemma implies that $K(\theta_1) \subseteq K(z_j^{*})$. 
Recall that we have just shown that $v(z_j^* - \sigma(z_j^*)) = v(z_j^*) = 1/(d_j - 1)$ for all $\sigma \in G$. 
Moreover, we can just switch $z_j^*$ and $\theta_1$ to arrive at the inequality
\[
v(\theta_1 - z_j^*) > v(z_j^* - \sigma(z_j^*)) = 1/(d_j - 1),
\]
and so we can apply Krasner's lemma again to deduce that $K(z_j^*) \subseteq K(\theta_1)$. 
Therefore, we have shown that $K(P)\cong K(z_j^*)\cong K(\theta_1)$ and this extension is totally ramified of degree $d_j- 1$, hence tamely ramified. We also have that $z_j^*$, which is a linear combination of $x_1,\dots, x_g$ with unit coefficients, is a uniformizer for $K(P)$ and hence $\cO_{K(P)} \cong \cO_K[x_1,\dots,x_g]$. Finally, we note that there exists some coordinate $x_i$ of $P$ which has valuation $v(x_i) = v(z_j^*)$ by condition (3) of \autoref{lemma:niceisom} and hence $x_i$ is a uniformizer for $K(P)$ as well. 

The second statement of \autoref{thm:tamelyramifiedtorsion} follows because $K(\sF[p])$ is the compositum of the fields $K(P)/K$ as $P$ varies over points in $\sF[p]$ and the compositum of tamely ramified extensions is again tamely ramified. 
\end{proof}

  \bibliography{refs}{}
\bibliographystyle{amsalpha}

 \end{document}